\documentclass[12pt,bezier]{article}
\usepackage{amssymb}
\usepackage{mathrsfs}
\usepackage{amsmath}
\usepackage{amsfonts,amsthm,amssymb}
\usepackage{amsfonts}
\usepackage{graphicx}
\usepackage{caption}
\usepackage{float}
\usepackage[numbers,sort&compress]{natbib}

\newtheorem{lem}{Lemma}[section]
\newtheorem{thm}[lem]{Theorem}

\begin{document}

\title{A $C_{4}$-decomposition of the $\lambda$-fold line graph of $K_{m,n}$}

\author{Yifan Zhao, Yulong Wei, Weihua Yang\footnote{Corresponding author. E-mail: ywh222@163.com; yangweihua@tyut.edu.cn}\\
\\ \small Department of Mathematics, Taiyuan University of
Technology,\\
\small  Taiyuan Shanxi-030024,
China\\
}
\date{}
\maketitle

{\small{\bf Abstract:}
The small cycle decompositions of line graph ($\lambda$-fold line graph) of complete graphs and complete bipartite graphs are studied by many papers. In particular, Colby and Rodger obtained necessary and sufficient conditions for the existence of a $C_{4}$-decomposition of the $\lambda$-fold line graph of $K_{n}$,  and Ganesamurthy and Paulraja completely determined the values of $n$ and $\lambda$ for which the $\lambda$-fold line graph of $K_{n}$ has a $C_{5}$-decomposition. In this paper, we obtain the necessary and sufficient condition for the existence of a $C_{4}$-decomposition of the $\lambda$-fold line graph of $K_{m,n}$.

\vskip 0.5cm  Keywords: Cycle decomposition; Line graph; Complete bipartite graph

\section{Introduction}

All graphs considered here are loopless and finite. A graph $G$ is a pair $(V(G),E(G))$ consisting of a set $V(G)\neq\emptyset$ and a set $E(G)$ of
two-element subsets of $V(G)$, where $V(G)$ and $E(G)$ are called the vertex set and edge set of $G$ respectively. If $G$ is a multigraph, the edges with two same endpoints are regarded as different edges.
For nonempty subsets $S\subseteq V(G)$ and $T\subseteq E(G)$, we use $G[S]$
and $G-T$ to denote the subgraph of $G$ induced by $S$ and the subgraph of $G$ obtained by deleting $T$, respectively. If $G$ is a simple graph, let $\lambda G$ denote the multigraph in which vertices $u$ and $v$ are joined by $\lambda$ edges if they are adjacent in $G$. The cartesian product of simple graphs $G$ and $H$ is the graph $G\Box H$ whose vertex set is $V(G)\times V(H)$ and whose edge set is the set of all pairs $\{(u_{1},v_{1}),(u_{2},v_{2})\}$ such that either $\{u_{1},u_{2}\}\in E(G)$ and $v_{1}=v_{2}$, or $\{v_{1},v_{2}\}\in E(H)$ and $u_{1}=u_{2}$. The union of two graphs $G_{1}$ and $G_{2}$, denoted by $G_{1}\cup G_{2}$, is the graph with vertex set $V(G_{1})\cup V(G_{2})$ and edge set $E(G_{1})\cup E(G_{2})$. The join of two disjoint graphs $G$ and $H$, denoted by $G\vee H$, is the graph obtained from $G\cup H$ by joining each vertex of $G$ to each vertex of $H$. Let $K_{n}$ be the complete graph on $n$ vertices with vertex set $Z_{n}=\{0,1,\ldots,n-1\}$ and edge set $E(K_{n})=\{\{i,j\}\mid i,j\in Z_{n}, i\neq j\}$. Let $K_{m,n}$ be the complete bipartite graph $(X, Y)$ with $|X|=m$ and $|Y|=n$. Let $C_{n}$ be the cycle on $n$ vertices denoted by $(v_{1},v_{2},\ldots,v_{n})$.

The line graph of a graph $G$, denoted by $L(G)$, is the graph with $V(L(G))=E(G)$ and $\{e_{1},e_{2}\}\in E(L(G))$ if and only if $|e_{1}\cap e_{2}|=1$ in $G$. The $j$-clique $G(j)$ in $L(G)$ is the
subgraph of $L(G)$ induced by the vertex set $\{\{i,j\}\mid \{i,j\}\in E(G)\}$. In particular, $K_{m,n}(j)$ is the $j$-clique in $L(K_{m,n})$. Clearly, $L(K_{m,n})$ is isomorphic to $K_{m}\Box K_{n}$.

If $H_{1},H_{2},\ldots,H_{l}$ are edge disjoint subgraphs of $G$ such that $E(G)=\cup_{i=1}^l E(H_{i})$, then we say that $G$ has an $\{H_{1},H_{2},\ldots,H_{l}\}$-decomposition, and we write this as $G=H_{1}\oplus H_{2}\oplus \cdots \oplus H_{l}$ or $G=\oplus_{i=1}^l H_{i}$. For simplicity of presentation, we define $G\oplus \emptyset=G$. If $H_{i}\cong H$ for $1\leq i\leq l$, then we say that $G$ has an $H$-decomposition. If $H\cong C_{k}$, then we say that $G$ admits a $C_{k}$-decomposition. Obviously, if $G$ has a $C_{k}$-decomposition, then $\lambda G$ has a $C_{k}$-decomposition. For undefined graph theoretical terms we refer readers to~\cite{Bondy}.

The cycle decompositions of the line graph and $\lambda$-fold line graph of the complete graphs were studied by many authors.  Heinrich and Nonay \cite{Heinrich} characterized the 4-cycle decompositions of the line graph and $\lambda$-fold line graph of the complete graph; Colby and Rodger \cite{Colby} characterized that $\lambda L(K_{n})$ has a $C_{4}$-decomposition if and only if $n$ and $\lambda$ satisfy $(1)$ $n$ is even, or $(2)$ $n\equiv 1~(mod~4)$ and $\lambda \equiv 0~(mod~2)$, or $(3)$ $n\equiv 3~(mod~4)$ and $\lambda \equiv 0~(mod~4)$, or $(4)$ $n\equiv 1~(mod~8)$ and $\lambda$ is odd.  Cox \cite{Cox1} obtained necessary and sufficient conditions for the existence of a $C_{6}$-decomposition of $L(K_{n})$. For $n\equiv 1~(mod~2m)$, or $n\equiv 0$ or $2~(mod~ m)$ ($ m=2^{i}$), Cox and Rodger \cite{Cox2} showed that the obvious necessary condition for decomposing $L(K_{n})$  into $m$-cycles is also sufficient.  Recently, Ganesamurthy and Paulraja \cite{Ganesamurthy} characterized the condition for  $\lambda L(K_{n})$ to have a $C_{5}$-decomposition for $\lambda=1$ and $\lambda>1$ .

For the line graph of the complete bipartite graphs, Hoffman and Pike \cite{Hoffman} proved that $L(K_{m,n})$ has a $C_{4}$-decomposition if and only if one of the following conditions is satisfied: $(1)$ $m,n\equiv 0~(mod~2)$, $(2)$ $m,n\equiv 1~(mod~8)$, or $(3)$ $m,n\equiv 5~(mod~8)$.  In this paper, we obtain the necessary and sufficient condition for the existence of a $C_{4}$-decomposition of the $\lambda$-fold line graph of $K_{m,n}$ as follows.

\begin{thm}\label{thm1.1}
For all positive integers $m,n$ and $\lambda$ with $mn\geq 4$, the $\lambda$-fold line graph of $K_{m,n}$ has a $C_{4}$-decomposition if and only if

$(1)$ $2\mid \lambda (m+n-2)$,

$(2)$ $8\mid \lambda mn(m+n-2)$,

$(3)$ if $m$~$(resp. ~n)$~$\equiv 3~(mod~8)$ and $n$~$(resp. ~m)$~$\equiv 7~(mod~8)$, then $\lambda \equiv 0~(mod~2)$, and

$(4)$ if $m$~$(resp. ~n)$~$=2$ and $n$~$(resp. ~m)$~$\equiv 1~(mod~2)$, then $n$~$(resp. ~m)\geq5$ and $\lambda \equiv 0~(mod~8)$.
\end{thm}

\section{Preliminaries}

Here, we present several known results on the cycle decompositions of complete graphs and complete bipartite graphs that will be used in the proof of the main result.

\begin{thm}[\cite{Hoffman}]\label{thm2.1}
The graph $K_{m}\Box K_{n}$ with $mn\geq 4$ has a $C_{4}$-decomposition if and only if

$(1)$ $m, n\equiv 0~(mod~2)$,

$(2)$ $m, n\equiv 1~(mod~8)$, or

$(3)$ $m, n\equiv 5~(mod~8)$.
\end{thm}

\begin{thm}[\cite{Bryant}]\label{thm2.2}
Let $\lambda$, $n$ and $m$ be integers with $n$, $m\geq 3$ and $\lambda \geq 1$. There exists a $C_{m}$-decomposition of $\lambda K_{n}$  if and only if:

$(1)$ $m\leq n$,

$(2)$ $\lambda (n-1)$ is even,

$(3)$ $m$ divides $\lambda$$n\choose2$.
\end{thm}

\begin{thm}[\cite{Fu}]\label{thm2.3}
Let $n$ be a positive integer, and let $F$ be a $2$-regular subgraph of $K_{n}$. There exists a $C_{4}$-decomposition of $K_{n}-E(F)$ if and only if $n$ is odd and $4$ divides the number of edges in $K_{n}-E(F)$.
\end{thm}

\begin{thm}[\cite{Sehgal}]\label{thm2.4}
There exists  a set $F=\{F(1),\ldots,F(4)\}$ of four edge-disjoint $1$-factors in $K_{8x}$ such that there exists a $C_{4}$-decomposition of $(K_{8x}-F)\vee K_{1}$.
\end{thm}

\begin{thm}[\cite{Lee}]\label{thm2.5}
For positive integers $\lambda, k, m$ and $n$ with $\lambda m \equiv \lambda n\equiv k\equiv 0~(mod~2)$ and $min\{m,n\}\geq k/2\geq 2$, the multigraph $\lambda K_{m,n}$ has a $C_{k}$-decomposition if one of the following conditions holds:

$(1)$~$\lambda$ is odd and $k$ divides $mn$,

$(2)$~$\lambda$ is even and $k$ divides $2mn$, or

$(3)$~$\lambda$ is even and $\lambda n$ or $\lambda m$ is divisible by $k$.
\end{thm}

\section{Cycle decompositions for small $m$, $n$, $\lambda$}
In this section, we obtain $C_{4}$-decompositions of $\lambda L(K_{m,n})$ for some small  $m$, $n$ and $\lambda$.

\begin{lem}\label{lem3.1}
If $n\equiv 1~(mod~2)$ and $n\geq3$, then $2L(K_{n,n})$ has a $C_{4}$-decomposition.
\end{lem}

\begin{proof}
Let $K_{n,n}$ be defined with bipartition $X$ and $Y$. We suppose that $K_{X}$ and $K_{Y}$ are the complete graphs on vertex sets $X$ and $Y$, respectively. Since $n\equiv 1~(mod~2)$, $K_{X}$ and $K_{Y}$ have Hamilton decompositions. Let $A=\{A_{u}\mid1\leq u\leq \frac{n-1}{2}\}$ and $B=\{B_{u}\mid1\leq u\leq \frac{n-1}{2}\}$ be Hamilton decompositions of $K_{X}$ and $K_{Y}$. We assign an orientation to the cycles $A_{u}\in A$ and $B_{u}\in B$ for each $1\leq u\leq \frac{n-1}{2}$, and denote those by $\overrightarrow{A_{u}}$ and $\overrightarrow{B_{u}}$. For each $1\leq u\leq \frac{n-1}{2}$, we define $\overrightarrow{A_{u}}\times \overrightarrow{B_{u}}=\{(\{a,x\},\{x,b\},\{b,y\},\{y,a\})\mid(a,b)\in \overrightarrow{A_{u}}, (x,y)\in \overrightarrow{B_{u}}\}$ and $C=\cup_{u=1}^{\frac{n-1}{2}}\overrightarrow{A_{u}}\times \overrightarrow{B_{u}}$. Let $E(C)$ be the set of edges of $C$. For any edge $\{\{i,j\},\{j,k\}\}$ of $2L(K_{n,n})$ with $i, k\in X$, $j\in Y$, there  exists a  $\overrightarrow{A_{u}}$ such that an arc $(i,k)$ or $(k,i)$ lies in $\overrightarrow{A_{u}}$, and there exist vertices $s,t\in Y$ such that arcs $(s,j)$, $(j,t)$ lie in $\overrightarrow{B_{u}}$. If $(i,k)$ lies in $\overrightarrow{A_{u}}$, then $(\{i,j\},\{j,k\},\{k,t\},\{t,i\})\in C$. If $(k,i)$ lies in $\overrightarrow{A_{u}}$, then $(\{k,s\},\{s,i\},\{i,j\},\{j,k\})\in C$. Similarly, any edge $\{\{i,j\},\{j,k\}\}$ of $2L(K_{n,n})$ with $i, k\in Y$, $j\in X$ belongs to $E(C)$. Thus, $E(2L(K_{n,n}))\subseteq E(C)$. Note that $|E(2L(K_{n,n}))|=2n^{2}(n-1)=|E(C)|$. Hence, $2L(K_{n,n})$ has a $C_{4}$-decomposition.
\end{proof}

\begin{lem}\label{lem3.2}
The graphs $2L(K_{3,4})$ and $2L(K_{3,7})$ have a $C_{4}$-decomposition.
\end{lem}

\begin{proof}
Let $V(2L(K_{3,4}))=Z_{3}\times Z_{4}$. The following is a $C_{4}$-decomposition of $2L(K_{3,4})$:
$((0,0),(0,1),\\(1,1),(1,0))$, $((0,0),(0,2),(1,2),(1,0))$,
$((0,2),(0,3),(1,3),(1,2))$, $((0,1),(0,3),(1,3),(1,1))$,
$((1,0),\\(1,2),(2,2),(2,0))$, $((1,0),(1,3),(2,3),(2,0))$,
$((1,1),(1,2),(2,2),(2,1))$, $((1,1),(1,3),(2,3),(2,1))$,\\
$((0,0),(0,1),(2,1),(2,0))$, $((0,0),(0,3),(2,3),(2,0))$,
$((0,1),(0,2),(2,2),(2,1))$, $((0,2),(0,3),(2,3),\\(2,2))$,
$((0,0),(0,2),(0,1),(0,3))$, $((1,0),(1,1),(1,2),(1,3))$,
$((2,0),(2,1),(2,3),(2,2))$.

Suppose $X=\{1,2,3\}$, $Y_1=\{4,5,6,7\}$ and $Y_2=\{8,9,10\}$. Let $K_{3,7}=(X, Y_1\cup Y_2)$, $A_i=\{\{i,j\}\in E(K_{3,7})\mid j\in Y_1\cup Y_2\}$ for $1\leq i\leq3$, and $G_k=2L(K_{3,7}[X\cup Y_k])$ for $1\leq k\leq2$.
The induced subgraph of $(2L(K_{3,7})-E(G_1\cup G_2))[A_i]$ is denoted by $H_i$ for $1\leq i\leq3$. Then $2L(K_{3,7})=(\oplus_{k=1}^2 G_k)\oplus(\oplus_{i=1}^3 H_i)$. Note that $G_1\cong 2L(K_{3,4})$, $G_2\cong2L(K_{3,3})$ and
$H_{i}\cong 2K_{3,4}$ for $1\leq i\leq3$. Thus, by Lemma \ref{lem3.1} and Theorem \ref{thm2.5}, $2L(K_{3,7})$ has a $C_{4}$-decomposition.
\end{proof}

\begin{lem}\label{lem3.3}
The graph $4L(K_{3,5})$ has a $C_{4}$-decomposition.
\end{lem}

\begin{proof}
Let $V(4L(K_{3,5}))=Z_{3}\times Z_{5}$. Note that $4L(K_{3,5})\cong4(K_{3}\square K_{5})$. For $0\leq i\leq 4$, let $C_{1}^{i}=C_{4}^{i}=((0,i),(0,1+i),(1,1+i),(1,i))$, $C_{2}^{i}=C_{5}^{i}=((1,i),(1,2+i),(2,2+i),(2,i))$, $C_{3}^{i}=((0,i),(0,1+i),(2,1+i),(2,i))$ and $C_{6}^{i}=((0,i),(0,2+i),(2,2+i),(2,i))$ be edge disjoint $4$-cycles of $4L(K_{3,5})$. Let $C=\cup _{j=1}^{6}(\cup _{i=0}^{4}C_{j}^{i})$. Clearly, $4L(K_{3,5})-E(C)=H_{1}\oplus H_{2}\oplus H_{3}$, where $V(H_{1})=\{0\}\times Z_{5}$, $V(H_{2})=\{1\}\times Z_{5}$, $V(H_{3})=\{2\}\times Z_{5}$, $H_{1}\cong H_{3}\cong K_{5}+E(2C_5)$ and $H_{2}\cong 2K_{5}$. Note that $2K_{5}$ admits a $C_{4}$-decomposition by Theorem \ref{thm2.2}. We give a $C_{4}$-decomposition of $H_{1}$ below: $((0,0),(0,2),(0,4),(0,1))$, $((0,1),(0,3),(0,0),(0,2))$, $((0,2),(0,4),(0,1),(0,3))$, $((0,0),(0,2),(0,4),(0,3))$,$((0,0),(0,3),(0,1),(0,4))$.

Hence, $4L(K_{3,5})$ admits a $C_{4}$-decomposition.
\end{proof}

\begin{lem}\label{lem3.4}
The graph $8L(K_{2,n})$ has a $C_{4}$-decomposition, where $n\equiv 1~(mod~2)$ and $n\geq5$.
\end{lem}

\begin{proof}
Let $V(8L(K_{2,5}))=Z_{2}\times Z_{5}$. For $0\leq i\leq 4$, let $C_{1}^{i}=C_{3}^{i}=((0,i),(0,1+i),(1,1+i),(1,i))$ and $C_{2}^{i}=C_{4}^{i}=((0,i),(0,2+i),(1,2+i),(1,i))$ be edge disjoint $4$-cycles of $8L(K_{2,5})$. Let $C=\cup _{j=1}^{4}(\cup _{i=0}^{4}C_{j}^{i})$. Clearly, $8L(K_{2,5})-E(C)=H_{1}\oplus H_{2}$, where $V(H_{1})=\{0\}\times Z_{5}$, $V(H_{2})=\{1\}\times Z_{5}$ and $H_{1}\cong H_{2}\cong 6K_{5}$. Note that $6K_{5}$ admits a $C_{4}$-decomposition by Theorem \ref{thm2.2}. Thus, $8L(K_{2,5})$ admits a $C_{4}$-decomposition. If $n\geq7$ and $n\equiv 1~(mod~2)$, then
$8L(K_{2,n})=8L(K_{2,5})\oplus 8L(K_{2,n-5})\oplus 8K_{5,n-5}\oplus 8K_{5,n-5}$. Since $8L(K_{2,n-5})\cong 8(K_{2}\Box K_{n-5})$, by Theorems \ref{thm2.1} and \ref{thm2.5}, $8L(K_{2,n})$ has a $C_{4}$-decomposition.
\end{proof}

\begin{lem}\label{lem3.5}
The graphs $4L(K_{3,6})$ and $4L(K_{3,9})$ have $C_{4}$-decompositions.
\end{lem}

\begin{proof}
Suppose $X=\{1,2,3\}$, $Y_1=\{4,5,6\}$ and $Y_2=\{7,8,9\}$. Let $K_{3,6}=(X, Y_1\cup Y_2)$, $A_i=\{\{i,j\}\in E(K_{3,6})\mid j\in Y_1\cup Y_2\}$ for $1\leq i\leq3$, and $G_k=4L(K_{3,6}[X\cup Y_k])$ for $1\leq k\leq2$.
The induced subgraph of $(4L(K_{3,6})-E(G_1\cup G_2))[A_i]$ is denoted by $H_i$ for $1\leq i\leq3$. Then $4L(K_{3,6})=(\oplus_{k=1}^2 G_k)\oplus(\oplus_{i=1}^3 H_i)$. Note that $G_1\cong G_2\cong4L(K_{3,3})$ and
$H_{i}\cong 4K_{3,3}$ for $1\leq i\leq3$. Thus, by Lemma \ref{lem3.1} and Theorem \ref{thm2.5}, $4L(K_{3,6})$ has a $C_{4}$-decomposition.

Note that $4L(K_{3,9})=4L(K_{3,3})\oplus 4L(K_{3,6})\oplus 4K_{3,6}\oplus 4K_{3,6}\oplus 4K_{3,6}$. Thus, $4L(K_{3,9})$ has a $C_{4}$-decomposition by Lemma \ref{lem3.1} and Theorem \ref{thm2.5}.
\end{proof}

\section{Proof of Theorem~\ref{thm1.1}}
We prove our main result after some useful lemmas.
\begin{lem}\label{lem4.1}
For positive integers $a$, $m$, $n$ and $k$ with $n\geq 4$, if $2L(K_{m,4k})$ has a $C_{4}$-decomposition, then $2L(K_{m+4a,4k})$ has a $C_{4}$-decomposition; if $4L(K_{m,n})$ has a $C_{4}$-decomposition, then $4L(K_{m+4a,n})$ has a $C_{4}$-decomposition.
\end{lem}
\begin{proof}
Suppose $X_1=\{s(i)\mid 1\leq i\leq m\}$, $X_2=\{v(i)\mid 1\leq i\leq 4a\}$ and $Y=\{t(i)\mid 1\leq i\leq 4k\}$.
Let $K_{m+4a,4k}=(X_1\cup X_2, Y)$, $A_j=\{\{t(j),s(h)\}\in E(K_{m+4a,4k})\mid 1\leq h\leq m\}$
and $B_j=\{\{t(j),v(i)\}\in E(K_{m+4a,4k})\mid 1\leq i\leq 4a\}$ for $1\leq j\leq4k$.

{\bf Case 1.}\ $m=1$.

Let $G_0=2K_{1+4a,4k}(s(1))$, $G_{1,i}=2K_{1+4a,4k}(v(i))$ for $1\leq i\leq4a$, and $G_{2,j}=2K_{1+4a,4k}(t(j))$ for $1\leq j\leq4k$.
Then $2L(K_{1+4a,4k})=G_0\oplus(\oplus_{i=1}^{4a} G_{1,i})\oplus(\oplus_{j=1}^{4k} G_{2,j})$. Note that $G_0\cong 2L(K_{1,4k})$ and $G_{1,i}\cong 2K_{4k}$ for $1\leq i\leq4a$ and
$G_{2,j}\cong 2K_{4a+1}$ for $1\leq j\leq 4k$. Thus, by Theorem \ref{thm2.2}, $2L(K_{1+4a,4k})$ has a $C_{4}$-decomposition.

{\bf Case 2.}\ $m\geq2$.

Let $G_0=2L(K_{m+4a,4k}[X_1\cup Y])$, $G_{1,i}=2K_{m+4a,4k}(v(i))$ for $1\leq i\leq4a$, $G_{2,j}=2K_{m+4a,4k}[X_2\cup Y](t(j))$ and $G_{3,j}=(2L(K_{m+4a,4k})-E(G_0\cup(\cup_{i=1}^{4a}G_{1,i})\cup(\cup_{i=1}^{4k}G_{2,i})))[A_j\cup B_j]$ for $1\leq j\leq4k$.
Then $2L(K_{m+4a,4k})=G_0\oplus(\oplus_{i=1}^{4a} G_{1,i})\oplus(\oplus_{j=1}^{4k} G_{2,j})\oplus(\oplus_{j=1}^{4k} G_{3,j})$. Note that $G_0\cong2L(K_{m,4k})$, $G_{1,i}\cong 2K_{4k}$ for $1\leq i\leq4a$,
$G_{2,j}\cong 2K_{4a}$ and $G_{3,j}\cong 2K_{4a,m}$ for $1\leq j\leq 4k$. Thus, by Theorems \ref{thm2.2} and \ref{thm2.5}, $2L(K_{m+4a,4k})$ has a $C_{4}$-decomposition.

Similarly, $4L(K_{1+4a,n})=4L(K_{1,n})\oplus\underbrace{4K_{n}\oplus\cdots\oplus4K_{n}}_{4a~copies}\oplus\underbrace{4K_{4a+1}\oplus\cdots\oplus4K_{4a+1}}_{n~copies}$, when $m=1$. Thus, the graph $4L(K_{1+4a,n})$ has a $C_{4}$-decomposition by Theorem 2.2. When $m\geq2$, $4L(K_{m+4a,n})=4L(K_{m,n})\oplus\underbrace{4K_{n}\oplus\cdots\oplus4K_{n}}_{4a~copies}\oplus\underbrace{4K_{4a}\oplus\cdots\oplus4K_{4a}}_{n ~copies}\oplus\underbrace{4K_{4a,m}\oplus\cdots\oplus4K_{4a,m}}_{n~copies}$, and then $4L(K_{m+4a,n})$ has a $C_{4}$-decomposition by Theorems 2.2 and 2.5.
\end{proof}

\begin{lem}\label{lem4.2}
For positive integers $\lambda$, $m$, $n$ and $a$ with $m\equiv n\equiv 1~(mod~2)$ and $mn\geq 4$, if $\lambda L(K_{m,n})$ has a $C_{4}$-decomposition, then $\lambda L(K_{m+8a,n})$ has a $C_{4}$-decomposition.
\end{lem}

\begin{proof}
Suppose $X_1=\{s(i)\mid 1\leq i\leq m\}$, $X_2=\{v(i)\mid 1\leq i\leq 8a\}$ and $Y=\{t(i)\mid 1\leq i\leq n\}$. Let $K_{m+8a,n}=(X_1\cup X_2, Y)$. We suppose that $K_{8a}$ is the complete graph on the vertex set $X_2$.
By Theorem \ref{thm2.4}, there exists a set $F=\{F(i)|1\leq i\leq 4\}$ of four edge-disjoint $1$-factors in $K_{8a}$ such that $(K_{8a}-E(\cup_{i=1}^4F(i)))\vee K_{1}$  has a $C_{4}$-decomposition.

{\bf Case 1.}\ $n\neq 1$ and $n\neq 5$.

We consider the following edge-disjoint subgraphs of $\lambda L(K_{m+8a,n})$.

(1) Suppose $G_1=\lambda L(K_{m+8a,n}[X_1\cup Y])$. Since $G_1\cong \lambda L(K_{m,n})$, $G_1$ has a $C_{4}$-decomposition.

(2) For each $1\leq i\leq 8a$, let $C_{l}(i)$ be the cycle $(\{v(i),t(1)\},\{v(i),t(2)\},\ldots,\{v(i),t(l)\})$ of $K_{m+8a,n}(v(i))$ and let $G_{2,i}=\lambda(K_{m+8a,n}(v(i))-E(C_{l}(i)))$. If $l=0$, we set $E(C_{l}(i))=\emptyset$. If $n\equiv 1~(mod~8)$ and $l=0$, then $G_{2,i}\cong\lambda K_{n}$ has a $C_{4}$-decomposition by Theorem \ref{thm2.2}. If $n\equiv 3,5$ or $7~(mod~8)$ and $l=3, 6$ or $5$ respectively, and $n\neq3$,
then $G_{2,i}$ has a $C_{4}$-decomposition by Theorem \ref{thm2.3}. If $n=3$, then let $\oplus_{i=1}^{8a}G_{2,i}=\emptyset$.

(3) For $l>0$ and each $1\leq i\leq 8a$, we define an edge-coloring of $\lambda C_{l}(i)$ with at most three colors as follows. The edges in $\lambda C_{l}(i)$ joining each pair of adjacent vertices are assigned by the same color. Next, we only need to give the edge-coloring of $C_{l}(i)$. If $l=6$, the edges $\{\{v(i),t(1+j)\},\{v(i),t(2+j)\}\}$ for $j\in\{0,2,4\}$ are assigned by color $x_1$ and all the other edges are assigned by color $x_2$. If $l\in\{3,5\}$, then the edge $\{\{v(i),t(1)\},\{v(i),t(l)\}\}$ is assigned by color $x_3$, and the remaining edges of $C_{l}(i)$ are alternately assigned by $x_1$ and $x_2$, where the edge $\{\{v(i),t(1)\},\{v(i),t(2)\}\}$ is assigned by color $x_1$.

Assume that $e_{i,j}=\{\{v(i),t(j)\},\{v(i),t(j+1)\}\}\in E(C_{l}(i))$ with color $x_k$ and the vertices $v(i)$ and $v(i_{k})$ are adjacent in $F(k)$. Note that the edges $e_{i,j}\in E(C_{l}(i))$ and $e_{i_{k},j}\in E(C_{l}(i_{k}))$ have the same color $x_k$ and form a cycle $H_{i,j}=(\{v(i),t(j)\},\{v(i),t(j+1)\},\{v(i_{k}),t(j+1)\},\{v(i_{k}),t(j)\})$ for some $1\leq j\leq l-1$ or $H_{i,l}=(\{v(i),t(l)\},\{v(i),t(1)\},\{v(i_{k}),t(1)\},\{v(i_{k}),t(l)\})$ for $j=l$. Let $G_3=\cup_{i=1}^{8a}(\cup_{j=1}^{l}H_{i,j})$ and $H^{\star}=\cup_{i=1}^{8a}C_{l}(i)$. Thus, $H^{\star}\subseteq G_3$ and $\lambda G_3$ has a $C_{4}$-decomposition. If $l=0$, then let $\lambda G_3=\emptyset$.

(4) Note that $G_3-E(H^{\star})$ is a spanning subgraph of $\cup_{j=1}^lK_{m+8a,n}[X_2\cup Y](t(j))$. In fact, for any vertex $\{v(i),t(j)\}$ of $G_3$, $\{v(i),t(j)\}$ lies exactly in two cycles from $\{H_{i,j}\mid 1\leq j\leq l\}$. Then $d_{G_3-E(H^{\star})}(\{v(i),t(j)\})=2$. So $G_3-E(H^{\star})$ is a $2$-factor of $\cup_{j=1}^lK_{m+8a,n}[X_2\cup Y](t(j))$. For each $1\leq j\leq l$, let $R(t(j))=(G_3-E(H^{\star}))[V(K_{m+8a,n}[X_2\cup Y](t(j)))]$. Therefore, $R(t(j))$ is a $2$-factor of $K_{m+8a,n}[X_2\cup Y](t(j))$. Then, $\lambda G_3=\lambda H^{\star}\oplus\lambda (G_3-E(H^{\star}))=\lambda H^{\star}\oplus(\cup_{j=1}^{l}\lambda R(t(j)))$.
For $1\leq j\leq l$, let $G_{4,j}=\lambda(K_{m+8a,n}[\{s(m)\}\cup X_2\cup Y](t(j))-E(R(t(j))))$. By Theorem \ref{thm2.3}, $G_{4,j}$ has a $C_{4}$-decomposition for $1\leq j\leq l$. If $l=0$, then let $\oplus_{j=1}^{l}G_{4,j}=\emptyset$.

(5) For $l+1\leq j\leq n$, let $G_{5,j}=\lambda(K_{m+8a,n}[\{s(m)\}\cup X_2\cup Y](t(j)))$. Since $G_{5,j}\cong\lambda K_{8a+1}$, $G_{5,j}$ has a $C_{4}$-decomposition by Theorem \ref{thm2.2} for $l+1\leq j\leq n$.
If $n=3$, then let $\oplus_{j=l+1}^{n}G_{5,j}=\emptyset$.

(6) For $1\leq j\leq n$, suppose $A_j=\{\{t(j),s(i)\}|1\leq i\leq m-1\}\}$ and $B_j=\{\{t(j),v(i)\}|1\leq i\leq 8a\}$. Let $G_{6,j}=\lambda (K_{m+8a,n}(t(j))-\{t(j),s(m)\}-E(L(K_{m+8a,n})[A_j])-E(L(K_{m+8a,n})[B_j]))$.  Thus, $G_{6,j}\cong\lambda K_{m-1,8a}$ has a $C_{4}$-decomposition by Theorem \ref{thm2.5}. If $m=1$, let $\oplus_{j=1}^{n}G_{6,j}=\emptyset$.

Note that $\lambda L(K_{m+8a,n})=G_1\oplus(\oplus_{i=1}^{8a}G_{2,i})\oplus\lambda G_3\oplus(\oplus_{j=1}^{l}G_{4,j})\oplus(\oplus_{j=l+1}^{n}G_{5,j})\oplus(\oplus_{j=1}^{n}G_{6,j})$. Thus $\lambda L(K_{m+8a,n})$ has a $C_{4}$-decomposition.

{\bf Case 2.}\ $n=5$.

We consider the following edge-disjoint subgraphs of $\lambda L(K_{m+8a,5})$.

(1) Suppose $G_1=\lambda L(K_{m+8a,5}[X_1\cup Y])$. Since $G_1\cong \lambda L(K_{m,5})$, $G_1$ has a $C_{4}$-decomposition.

(2) For $1\leq i\leq 8a$, let $C(i)=(\{v(i),t(5)\},\{v(i),t(4)\},\{v(i),t(1)\})\cup (\{v(i),t(5)\},\{v(i),t(3)\},\\\{v(i),t(2)\})$ and $G_{2,i}=K_{m+8a,5}(v(i))-E(C(i))$. Then, the graph $G_{2,i}$ is the union of a $4$-cycle $(\{v(i),t(1)\},\{v(i),t(2)\},\{v(i),t(4)\},\{v(i),t(3)\})$ and a vertex $\{v(i),t(5)\}$. Thus, $\lambda G_{2,i}$ has a $C_{4}$-decomposition.

(3) For each $1\leq i\leq 8a$, we define an edge-coloring of $\lambda C(i)$ with four colors as follows. The edges in $\lambda C(i)$ joining each pair of adjacent vertices are assigned by the same color. Next, we only need to give the edge-coloring of $C(i)$. We color the edges $\{\{v(i),t(2)\},\{v(i),t(3)\}\}$ and $\{\{v(i),t(4)\},\{v(i),t(5)\}\}$ with $1$, $\{\{v(i),t(1)\},\{v(i),t(4)\}\}$ and $\{\{v(i),t(3)\},\{v(i),t(5)\}\}$ with $2$, $\{\{v(i),t(2)\},\{v(i),t(5)\}\}$ with $3$, $\{\{v(i),t(1)\},\{v(i),t(5)\}\}$ with $4$. Let $H_i$ be the set of cycles
$(\{v(i),t(2)\},\{v(i),t(3)\},\{v(i_{1}),t(3)\},\{v(i_{1}),t(2)\})$,
$(\{v(i),t(4)\},\{v(i),t(5)\},\{v(i_{1}),t(5)\},\\\{v(i_{1}),t(4)\})$,
$(\{v(i),t(1)\},\{v(i),t(4)\},\{v(i_{2}),t(4)\},\{v(i_{2}),t(1)\})$,
$(\{v(i),t(3)\},\{v(i),t(5)\},\\\{v(i_{2}),t(5)\},\{v(i_{2}),t(3)\})$,
$(\{v(i),t(2)\},\{v(i),t(5)\},\{v(i_{3}),t(5)\},\{v(i_{3}),t(2)\})$,
$(\{v(i),t(1)\},\\\{v(i),t(5)\},\{v(i_{4}),t(5)\},\{v(i_{4}),t(1)\})$,
where the vertex $v(i_{k})$ and the vertex $v(i)$ are adjacent in $F(k)$. Let $G_3=\cup_{i=1}^{8a}(\cup_{C\in H_i}C)$ and $H^{\star}=\cup_{i=1}^{8a}C(i)$. Then, $\lambda G_3$ has a $C_{4}$-decomposition.

(4) Similarly, $G_3$ exactly contains all edges of $H^{\star}$, a $2$-factor of $\cup_{j=1}^{4}K_{m+8a,5}[X_2\cup Y](t(j))$ and a $4$-factor of $K_{m+8a,5}[X_2\cup Y](t(5))$. For $1\leq j\leq 5$, let $R(t(j))=(G_3-E(H^{\star}))[V(K_{m+8a,5}[X_2\cup Y](t(j)))]$. Therefore, $R(t(j))$ is a $2$-factor of $K_{m+8a,5}[X_2\cup Y](t(j))$ for $1\leq j\leq 4$ and $R(t(5))$ is a $4$-factor of $K_{m+8a,5}[X_2\cup Y](t(5))$.
Note that $R(t(5))$ is isomorphic to $\cup_{k=1}^4 F(k)$.  Then, $\lambda G_3=\lambda H^{\star}\oplus\lambda (G_3-E(H^{\star}))=\lambda H^{\star}\oplus(\cup_{j=1}^{5}\lambda R(t(j)))$.
For $1\leq j\leq 4$, let $G_{4,j}=\lambda(K_{m+8a,5}[\{s(m)\}\cup X_2\cup Y](t(j))-E(R(t(j))))$. By Theorem \ref{thm2.3}, $G_{4,j}$ has a $C_{4}$-decomposition for $1\leq j\leq 4$.

(5) Let $G_{5}=K_{m+8a,5}[\{s(m)\}\cup X_2\cup Y](t(5))-E(R(t(5)))$. Since $G_{5}\cong (K_{8a}-E(\cup_{k=1}^4 F(k)))\vee K_{1}$, $\lambda G_{5}$ has a $C_{4}$-decomposition by Theorem \ref{thm2.4}.

(6) For $1\leq j\leq 5$, we can define the graph $G_{6,j}$ as Case 1 (6) satisfying that $G_{6,j}$ has a $C_{4}$-decomposition.

Note that $\lambda L(K_{m+8a,5})=G_1\oplus(\oplus_{i=1}^{8a}\lambda G_{2,i})\oplus\lambda G_3\oplus(\oplus_{j=1}^{4}G_{4,j})\oplus(\oplus\lambda G_{5})\oplus(\oplus_{j=1}^{5}G_{6,j})$. Thus $\lambda L(K_{m+8a,5})$ has a $C_{4}$-decomposition.

{\bf Case 3.}\ $n=1$.

Since $mn\geq4$, $m\geq4$. We consider the following edge-disjoint subgraphs of $\lambda L(K_{m+8a,1})$.

Suppose $G_1=\lambda L(K_{m+8a,1}[X_1\cup Y])$. Since $G_1\cong \lambda L(K_{m,1})$, $G_1$ has a $C_{4}$-decomposition.
Let $G_{2}=\lambda(K_{m+8a,1}[\{s(m)\}\cup X_2\cup Y](t(1))$. Since $G_{2}\cong\lambda K_{8a+1}$, $G_{2}$ has a $C_{4}$-decomposition by Theorem \ref{thm2.2}.
Suppose $A_1=\{\{t(1),s(i)\}|1\leq i\leq m-1\}\}$ and $B_1=\{\{t(1),v(i)\}|1\leq i\leq 8a\}$. Let $G_{3}=\lambda (K_{m+8a,1}(t(1))-\{t(1),s(m)\}-E(L(K_{m+8a,1})[A_1])-E(L(K_{m+8a,1})[B_1]))$. Since $G_{3}\cong\lambda K_{m-1,8a}$,  $G_{3}$ has a $C_{4}$-decomposition by Theorem \ref{thm2.5}.

Therefore, this lemma holds.
\end{proof}

Note that $\lambda L(K_{m,n})=\underbrace{\lambda K_{m}\oplus\cdots\oplus\lambda K_{m}}_{n ~copies}\oplus\underbrace{\lambda K_{n}\oplus\cdots\oplus\lambda K_{n}}_{m~copies}$.  A $4$-cycle in $\lambda L(K_{m,n})$ is $pure$ if all edges of the $4$-cycle are contained within one copy of $\lambda K_{m}$ or
one copy of $\lambda K_{n}$. A $4$-cycle in $\lambda L(K_{m,n})$ that is not pure is $mixed$. If all vertices of  $\lambda L(K_{m,n})$ are arranged in a rectangular grid with $m$ rows and $n$ columns, then a pure $4$-cycle contains four vertical edges or four horizontal edges and a mixed $4$-cycle contains two vertical edges and two horizontal edges. We are finally ready to prove Theorem 1.1.

\begin{proof}
 The obvious necessary conditions for the existence of a $C_{4}$-decomposition of $\lambda L(K_{m,n})$ are that each vertex in the graph has even degree and  $4$ divides the number of edges in the graph. We observe that $\lambda L(K_{m,n})$ has $mn$ vertices and the degree of each vertex is  $\lambda (m+n-2)$. Hence, $\lambda L(K_{m,n})$ has $\frac{\lambda mn(m+n-2)}{2}$ edges. Thus, we have  $2\mid \lambda (m+n-2)$ and $8\mid \lambda mn(m+n-2)$.

Assume that $m\equiv 3~(mod~8)$, $n\equiv 7~(mod~8)$ and $\lambda \equiv 1~(mod~2)$. Note that each pure $4$-cycle in $\lambda L(K_{m,n})$ contains even horizontal edges and each mixed $4$-cycle contains two horizontal edges. Hence, the number of horizontal edges used by all $4$-cycles is even. On the other hand, the total number of horizontal edges in $\lambda L(K_{m,n})$ is $\frac{\lambda mn(n-1)}{2}$. Note that $m$, $n$, $\lambda$ are odd, and  $\frac{n-1}{2}\equiv 3~(mod~4)$. Then the total number of horizontal edges in $\lambda L(K_{m,n})$ is odd, which is a contradiction. Thus, if $m$~$(resp. ~n)$~$\equiv 3~(mod~8)$ and $n$~$(resp. ~m)$~$\equiv 7~(mod~8)$, then $\lambda \equiv 0~(mod~2)$.

We consider the case of $m=2$, $n\equiv 1~(mod~2)$. Since $8\mid\lambda mn(m+n-2)$, we have $\lambda\equiv 0~(mod~4)$. Observe that each vertical edge is contained by a mixed $4$-cycle. Hence the total number of edges contained by all mixed $4$-cycles in each $\lambda K_{n}$ is $\frac{\lambda n}{2}$. The number of remaining edges in each $\lambda K_{n}$ is $\frac{\lambda n(n-2)}{2}$ and the remaining edges in each $\lambda K_{n}$ would be contained by pure $4$-cycles. Then $\lambda n(n-2)\equiv 0~(mod~8)$. On the other hand, if $\lambda \equiv 4~(mod~8)$, then $\lambda n(n-2)\equiv 4~(mod~8)$, which is a contradiction. If $n=3$, then $\lambda K_{3}$ has $3$ vertices and cannot contain 4-cycles, which is also a contradiction. Thus, if $m$~$(resp. ~n)$~$=2$ and $n$~$(resp. ~m)$~$\equiv 1~(mod~2)$, then $n$~$(resp.~m)\geq5$ and $\lambda \equiv 0~(mod~8)$.

Therefore, the necessity holds. We now prove sufficiency. Note that $L(K_{m,n})\cong K_{m}\Box K_{n}$. We discuss the following three cases.

{\bf Case 1.}\ $m$ and $n$ are both even.

By Theorem \ref{thm2.1}, $\lambda L(K_{m,n})$ has a $C_{4}$-decomposition.

{\bf Case 2.}\ $m$ and $n$ are both odd.

If $m$ and $n$ are both odd, then $m+n-2$ is even. Since $8\mid \lambda mn(m+n-2)$, we have $8\mid \lambda (m+n-2)$. Then, $m$, $n$, $\lambda$ satisfy  (1) $m+n-2\equiv 0~(mod~8)$, or (2)$m+n-2\equiv 2~(mod~4)$ and $\lambda\equiv 0~(mod~4)$, or (3)$m+n-2\equiv 4~(mod~8)$ and $\lambda\equiv 0~(mod~2)$.
Let $m\equiv k_{1}~(mod~8)$, $n\equiv k_{2}~(mod~8)$, where $0\leq k_{i}\leq7$ for $1\leq i\leq2$. Without loss of generality, we assume that $k_{1} \leq k_{2}$.

{\bf Case 2.1.}\  $m+n-2\equiv 0~(mod~8)$.

In this case, we have (1) $m,n\equiv 1~(mod~8)$ or (2) $m,n\equiv 5~(mod~8)$ or (3) $m\equiv 3~(mod~8)$, $n\equiv 7~(mod~8)$.
If $m$, $n$ satisfy (1) or (2), then $\lambda L(K_{m,n})$ has a $C_{4}$-decomposition by Theorem \ref{thm2.1}.
If $m\equiv 3~(mod~8)$, $n\equiv 7~(mod~8)$ with $\lambda \equiv 0~(mod~2)$, then $\lambda L(K_{m,n})$ has a $C_{4}$-decomposition by Lemmas \ref{lem3.2} and \ref{lem4.2}.

{\bf Case 2.2.}\  $m+n-2\equiv 2~(mod~4)$ and $\lambda\equiv 0~(mod~4)$.

Note that $mn\geq4$. If $m\equiv 1~(mod~8)$, $n\equiv 3~(mod~8)$ and $\lambda \equiv 0~(mod~4)$, then $4L(K_{1,11})\cong4K_{11}$ and $4L(K_{9,3})$ have a $C_{4}$-decomposition by Theorem \ref{thm2.2} and Lemma \ref{lem3.5}, respectively. Thus, by Lemma \ref{lem4.2}, $\lambda L(K_{m,n})$ has a $C_{4}$-decomposition.

If $m\equiv 1~(mod~8)$, $n\equiv 7~(mod~8)$ and $\lambda \equiv 0~(mod~4)$, then $4L(K_{1,7})\cong4K_7$ has a $C_{4}$-decomposition by Theorem \ref{thm2.2}. Thus, by Lemma \ref{lem4.2}, $\lambda L(K_{m,n})$ has a $C_{4}$-decomposition.

If $m\equiv 3~(mod~8)$, $n\equiv 5~(mod~8)$ and  $\lambda \equiv 0~(mod~4)$, then $4L(K_{3,5})$ has a $C_{4}$-decomposition by Lemma \ref{lem3.3}. Thus, by Lemma \ref{lem4.2}, $\lambda L(K_{m,n})$ has a $C_{4}$-decomposition.

If $m\equiv 5~(mod~8)$, $n\equiv 7~(mod~8)$ and  $\lambda \equiv 0~(mod~4)$, then $4L(K_{1,7})\cong4K_7$ has a $C_{4}$-decomposition by Theorem \ref{thm2.2}. Thus, by Lemmas \ref{lem4.1} and \ref{lem4.2}, $\lambda L(K_{m,n})$ has a $C_{4}$-decomposition.

{\bf Case 2.3.}\  $m+n-2\equiv 4~(mod~8)$ and $\lambda\equiv 0~(mod~2)$.

In this case, we have (1) $m\equiv 1~(mod~8)$, $n\equiv 5~(mod~8)$ or (2) $m, n\equiv 3~(mod~8)$ or (3) $m, n\equiv 7~(mod~8)$.
If $m$, $n$ satisfy (2) or (3) and $\lambda \equiv 0~(mod~2)$, then $2L(K_{3,3})$ and $2L(K_{7,7})$ have $C_{4}$-decompositions by Lemma \ref{lem3.1}. Thus, by Lemma \ref{lem4.2}, $\lambda L(K_{m,n})$ has a $C_{4}$-decomposition.
If $m\equiv 1~(mod~8)$, $n\equiv 5~(mod~8)$ and $\lambda \equiv 0~(mod~2)$, then $2L(K_{1,5})\cong 2K_5$ has a $C_{4}$-decomposition by Theorem \ref{thm2.2}. Thus, by Lemma \ref{lem4.2}, $\lambda L(K_{m,n})$ has a $C_{4}$-decomposition.

{\bf Case 3.}\ $m$ and $n$ have different parity.

Without loss of generality, we assume that $m$ is even and $n$ is odd. Since $2\mid \lambda (m+n-2)$ and $8\mid\lambda mn(m+n-2)$, we have $2\mid\lambda$ and $8\mid\lambda m$.

If $m\equiv 0~(mod~4)$, $n\equiv 1~(mod~4)$, then $2L(K_{m,1})\cong 2K_m$ has a $C_{4}$-decomposition by Theorem \ref{thm2.2}. Then $2L(K_{1+4a,m})$ has a $C_{4}$-decomposition by Lemma \ref{lem4.1}. Thus, $2L(K_{n,m})$ has a $C_{4}$-decomposition. Since $\lambda\equiv 0~(mod~2)$, we have $\lambda L(K_{m,n})$ has a $C_{4}$-decomposition.

If $m\equiv 0~(mod~4)$, $n\equiv 3~(mod~4)$, then $2L(K_{4,3})$ has a $C_{4}$-decomposition by Lemma \ref{lem3.2}. Then $2L(K_{3+4a,4})$ has a $C_{4}$-decomposition by Lemma \ref{lem4.1}. Thus, $2L(K_{n,4})$ has a $C_{4}$-decomposition. Suppose that $2L(K_{n,4(k-1)})$ has a $C_{4}$-decomposition with $k\geq2$. Note that $2L(K_{n,4k})=2L(K_{n,4(k-1)})\oplus 2L(K_{n,4})\oplus\underbrace{2K_{4(k-1),4}\oplus\cdots\oplus2K_{4(k-1),4}}_{n ~copies}$. By Theorem \ref{thm2.5}, $2L(K_{n,4k})$ has a $C_{4}$-decomposition. Thus, $\lambda L(K_{m,n})$ has a $C_{4}$-decomposition.

If $m\equiv 2~(mod~4)$, $n\equiv 3~(mod~4)$ and $m\geq6$, then $4L(K_{3,6})$ has a $C_{4}$-decomposition by Lemma \ref{lem3.5} and $4L(K_{3+4a,6})$ has a $C_{4}$-decomposition by Lemma \ref{lem4.1}. Thus, $4L(K_{6,n})$ has a $C_{4}$-decomposition. Suppose that $4L(K_{4(k-1)+2,n})$ has a $C_{4}$-decomposition with $k\geq2$. Note that $4L(K_{4k+2,n})=4L(K_{4(k-1)+2,n})\oplus 4L(K_{4,n})\oplus\underbrace{4K_{4(k-1)+2,4}\oplus\cdots\oplus4K_{4(k-1)+2,4}}_{n~copies}$. By Theorem \ref{thm2.5}, $4L(K_{4k+2,n})$ has a $C_{4}$-decomposition.
Since $\lambda\equiv 0~(mod~4)$, $\lambda L(K_{m,n})$ has a $C_{4}$-decomposition.

If $m\equiv 2~(mod~4)$, $n\equiv 1~(mod~4)$ and $m\geq6$, then $4L(K_{m,1})\cong 4K_m$ has a $C_{4}$-decomposition by Theorem \ref{thm2.2} and $4L(K_{1+4a,m})$ has a $C_{4}$-decomposition by Lemma \ref{lem4.1}. Thus, $4L(K_{n,m})$ has a $C_{4}$-decomposition. Since $\lambda\equiv 0~(mod~4)$, we have $\lambda L(K_{m,n})$ has a $C_{4}$-decomposition.

If $m=2$, $n\equiv 1~(mod~2)$, then $n\geq5$ and $\lambda\equiv 0~(mod~8)$. Thus, $\lambda L(K_{m,n})$ has a $C_{4}$-decomposition by Lemma \ref{lem3.4}.
\end{proof}

\end{document}